\documentclass[11pt]{amsart}
\usepackage{amsmath,amsfonts}
\usepackage{amssymb}
\usepackage{amscd}
\usepackage{amsthm}
\usepackage{yhmath}
\usepackage{subfigure}

\usepackage[all]{xy}
\usepackage{color}

\numberwithin{equation}{section}

\newtheorem{theorem}[equation]{Theorem}

\newtheorem{proposition}[equation]{Proposition}

\newtheorem{lemma}[equation]{Lemma}

\newtheorem{corollary}[equation]{Corollary}

\theoremstyle{remark}

\theoremstyle{definition}

\def\XXint#1#2#3{{\setbox0=\hbox{$#1{#2#3}{\int}$}
	\vcenter{\hbox{$#2#3$}}\kern-.5\wd0}}

\newcommand{\N}{\mathbb N}
\newcommand{\C}{\mathbb C}

\newcommand{\R}{\mathbb R}

\newcommand{\p}{\partial}

\newcommand{\e}{\epsilon}

\newcommand{\defun}{{\varphi}}

\newcommand{\al}{\alpha}

\def\eps{\epsilon}

\def\half{{1 \over 2}}

\newcommand{\Lip}{\operatorname{L}}

\begin{document}
\title[Conformal equivalence of visual metrics]{Conformal equivalence of  visual metrics in pseudoconvex domains}


\author{Luca Capogna}
 \thanks{L.C.  partially funded by NSF award DMS 1503683. }
\address{Department of Mathematical Sciences, Worcester Polytechnic Institute, Worcester, MA 01609}\email{lcapogna@wpi.edu}

\thanks{E.L.D.~is supported by the Academy of Finland project no. 288501. }

\author{Enrico Le Donne}
\address{
Department of Mathematics and Statistics, University of Jyv\"askyl\"a, 40014 Jyv\"askyl\"a, Finland}
\email{ledonne@msri.org}


\renewcommand{\subjclassname}{%
 \textup{2010} Mathematics Subject Classification}

\date{February 28, 2017}

\begin{abstract}  
We refine estimates introduced by Balogh and Bonk, to show that the boundary extensions of isometries between smooth strongly pseudoconvex domains in $\C^n$ are conformal with respect to the sub-Riemannian metric induced by the Levi form. As a corollary we obtain an alternative proof of a result of  Fefferman  on smooth extensions of biholomorphic mappings between pseudoconvex domains. The proofs are inspired by Mostow's proof of his rigidity theorem and are based on the asymptotic hyperbolic character of the Kobayashi or Bergman metrics and on the Bonk-Schramm hyperbolic fillings.
   \end{abstract}

\maketitle

\tableofcontents

\section{Introduction} 
Let $D \subset \mathbb C^n (n\geq2) $  be
a strongly pseudo-convex domain with $C^\infty$-smooth boundary. 
Denote by  $d_K$ the distance function corresponding  to a Finsler structure $K$
satisfying suitable estimates, see \eqref{bb-1.4}.  For example, one may consider  
  the  Bergman metric or the Kobayashi metric or the Carath\'eodory metric.
   In  \cite{Balogh-Bonk,MR1679982},  Balogh and Bonk have proved that the metric space $(D,d_K)$ is  hyperbolic in the sense of Gromov and its visual boundary coincides with the topological boundary $\p D$. They also show that the Carnot-Carath\'eodory metric  $d_{CC}$ corresponding to the Levi form on $\p D$,  determines the canonical class of snowflake equivalent visual metrics on  $\p D$. As a consequence, results from the theory of Gromov hyperbolic spaces can be immediately  applied in this setting. Among these we recall that every quasi-isometry between  such spaces extends to a quasi-conformal map between the visual boundaries, endowed with their families of visual metrics, see for instance \cite{MR1086655,Bridson:book} and references therein.

Our main contribution is to show that  extensions of isometries are actually  diffeomorphisms that are conformal with respect to the Carnot-Carath\'eodory metric. We only need to show that the extension is $1$-quasi-conformal, as the smoothness then follows from the recent results in \cite{CLDO}.

As in \cite{Balogh-Bonk}, our  strategy involves the  Bonk-Schramm hyperbolic filling metric $g$ defined in \eqref{filling-metric-cc}. This metric provides a stepping stone to connect the Carnot-Carath\'eodory distance, defined on the boundary by  the Levi form (see Section \ref{sec:Pseudoconvex}), with the invariant metric defined in the domain.

 \begin{theorem}\label{main} Let $D_1,D_2\subset \C^n$ be strongly pseudoconvex $C^\infty$-smooth domains and denote by $d_{K}$ the   distance function corresponding  to a Finsler structure $K$
satisfying     \eqref{bb-1.4}, and by $d_{CC}$ the Carnot-Carath\'eodory distance on the boundaries induced by the Levi form. If   $f:(D_1, d_{K}) \to (D_2,d_{K})$  is an isometry then the induced boundary map $F: (\p D_1, d_{CC})\to (\p D_2, d_{CC})$ is a diffeomorphism,  conformal with respect to the metric $d_{CC}$.
 \end{theorem}
 We emphasize that the result holds when  $d_K$ is 
 the Bergman, the Kobayashi, or the Carath\'eodory metrics. Indeed, these distances satisfy 
 \eqref{bb-1.4} in view of the  work in \cite{Balogh-Bonk,MR1679982,Ma}.

As we noted above, the proof of Theorem~\ref{main} is based on the study of the relation between the visual distances associated to $d_K$ and the visual distance of an ad-hoc hyperbolic filing metric, built through the Carnot-Carath\'eodory distance:
For $x\in D$ denote 
 by $h(x):=\sqrt {d_E(x,\partial D)}$
and
by $\pi(x) \in \partial D$ a  closest point in $\partial D$ with respect to the Euclidean distance $d_E(\cdot, \cdot)$, noting it is uniquely defined in a neighborhood of $\partial D$. 
  Set 
\begin{equation}\label{filling-metric-cc}
g(x,y):=2\log \bigg(   \frac{d_{\rm CC}(\pi(x), \pi(y))+ \max(h(x),h(y) ) }{\sqrt{h(x) h(y)} } \bigg).
\end{equation}
 This is an hyperbolic filling  metric built from the metric space $(\p D, d_{\rm CC})$ (see  Bonk and Schramm \cite{MR1771428}). Balogh and Bonk \cite[Corollary 1.3]{Balogh-Bonk}, showed  that $g$ is a metric in a neighborhood of $\p D$ and that $g$ and the invariant distance function $d_K$ are $(1,C)$-quasi-isometric.  As a consequence, they give rise to quasi-conformally equivalent visual metrics.
 
 The main technical point of our work is to refine this result in a quantitative fashion. We show that a particular visual quasi-distance $\rho^K_o$  associated to the invariant metric $d_K$
  is in fact  {\it pointwise and asymptotically} $(1+\e)$-quasi-conformally equivalent to the Carnot-Carath\'eodory $d_{CC}$ metric.
By {\it pointwise and asymptotically} we mean that  for every point $x\in \p D$ in the boundary, and for every $\e>0$, one can choose a base point $o$ for the definition of the visual distances so that the identity map has distortion less than $1+\e$ at $x$.
 Following ideas in CAT$(-1)$ spaces, given a pointed metric space $(X,d,o)$ we consider the {\it Bourdon
 function}
 \begin{equation}\label{def:Bourdon:dist}
 \rho^{d}_o ( x,y) = \exp (- \langle x, y \rangle_o) ,
 \end{equation}
 where $\langle x, y \rangle_o$ denotes the Gromov product in $(X,d)$, see Section~\ref{prelim}. 
 Usually, $ \rho^{d}_o $ is called Bourdon distance since for CAT($-1$) spaces it satisfies the triangle inequality. In our setting, 
$ \rho^{d}_o $ may not be a distance.

Moreover, on a CAT($-1$) space $X$   Bourdon showed in \cite{Bourdon-95}  that 
the visual boundaries $(\p_{\infty} X, \rho^d_o)$ corresponding to diffent base points  $o,o'\in X$ are conformally equivalent, thus implying immediately that any isometry of $ X $ extends to a conformal maps of its visual boundaries. 
Since pseudoconvex domains 
may not have negative curvature (see \cite{MR3114665}) and may not  be simply connected,
 they are not  CAT$(-1)$  spaces and so one cannot apply Bourdon's result.


Theorem~\ref{main} is achieved in two steps: First one shows that the Carnot-Carath\'eodory distance is conformally equivalent\footnote{ The result holds for any hyperbolic filling  as in the work of  Bonk and Schramm. See Section \ref{bbs}}
 to the Bourdon function $\rho^g_o$ associated to the hyperbolic filling metric $g$.
 
 \begin{proposition}\label{p1} For any $o\in D$, the functions $d_{CC}$ and $\rho^g_o$ are conformally equivalent.  \end{proposition}
 
 In other words, the identity map $ (\p D, d_{CC}) \to (\p D,\rho^g_o)$ has distortion  that is identically equal to one. See \eqref{hstella} for the definition of distortion.

Next, we show that at   every boundary point, and for any $\e>0$, one can find a base point $o\in D$ such that the corresponding visual functions $\rho^K_o$ and $\rho^g_o$ are $(1+\e)$-biLipschitz equivalent in a neighborhood of that point. In the following we denote Euclidean balls in $\C^n$ with the notation $B(x,r)$.


   \begin{proposition}\label{p2} 
   For any $\bar p\in  \p D$ and $\bar \e>0$ there exists 
$r>0$   
   such that for all
    $\omega \in \p D\cap B(\bar p,r)\setminus \{\bar p\}$ 
   there exists $r'>0$ such that for all 
   $o\in D\cap B(\omega,r')$
  the two functions $\rho^g_o$ and $ \rho^K_o$ are $(1+\bar \e)$-biLipschitz on $\p D \cap B(\bar p,r')$. 
 \end{proposition}

 The proof of Proposition~\ref{p2} and  Theorem~\ref{main} are in Section~\ref{loclip}. Theorem~\ref{main}  follows rather directly from Propositions~\ref{p1} and \ref{p2} and from the following diagram
 
  \begin{displaymath}\label{diagram}
\xymatrix{
	&	(D_1, d_{\rm K})  \ar@{~>}[d] 
\ar[r]^ {f}_{\rm {iso}}& (D_2, d_{\rm K})   \ar@{~>}[d] 
	&	\\
(\p D_1, \rho^{\rm g}_o) 
 \ar[r]^{\rm id}_{(1+\eps)\rm BL} 
&(\p D_1, \rho^{ K}_o)
	\ar[r]^F_{\rm iso}&	
	(\p D_2, \rho^{K}_{f(o)})
 \ar[r]^{\rm id}_{(1+\eps)\rm BL}  
&  
(\p D_2, \rho^{ g}_{f(o)})   
\ar[d]^{\rm id}_{ \rm conf}   	\\
    (\p D_1, d_{\rm CC} )	
  \ar[u]^{\rm id}_{  \rm  conf}
  \ar@{-->}[rrr]^F
&&	 
    & (\p D_2, d_{CC} ) 
      }\tag{D}
 \end{displaymath}

  At the center of this chain of compositions there is an isometry, the rest of the links are either $(1+\e)$ biLipschitz maps or conformal maps, so that the total distortion is at most $\e$ away from being  equal to 1 everywhere.

  From the  conformal equivalence theorem above and the results in \cite{CLDO}, one can immediately  infer a result about boundary extensions for biholomorphisms between strictly pseudoconvex domains in $\C^n$, originally established by Fefferman \cite{MR0350069}.
  
  \begin{corollary}
 Let $D_1,D_2\in \mathbb C^n (n\geq2) $  be
strongly pseudo-convex domains with $C^\infty$-smooth boundaries.  If $f:D_1\to D_2$ is a biholomorphism then it extends to a smooth map $F: \partial D_1 \to \partial D_2$ that is conformal with respect to the corresponding subRiemannian contact structure. In particular, at every boundary point, its differential is a similarity between the maximally complex tangent planes.
  \end{corollary}

Since the publication of  \cite{MR0350069} there have been several significative extensions and simplifications of the result. A small sample of this extensive line of inquiry can be found in the references \cite{bell, bell2, nirenberg, barrett, krantz}.

Rather than a simplification of Fefferman's original proof, our approach  is  a recasting of the result from the perspective of analysis in metric spaces and the circle of ideas at the core of Mostow rigidity \cite{Mostow}. The differentiable structure is not used to show that the extension map is $1$-quasi-conformal, and then it only enters in play coupled with the rigidity of $1$-quasi-conformal mappings in higher dimension. Likewise, curvature enters into the arguments only in its synthetic  (metric) form.  In particular, our work can be seen as an instance of a dictionary,  introduced by Bonk, Heinonen, and Koskela in \cite{BHK}, translating  back and forth problems in domains in Euclidean spaces by means of ad hoc hyperbolic or quasi-hyperbolic metrics, that endow such domains with an hyperbolic   structure in the sense of Gromov. For more results along this line, see also the recent, interesting work of Zimmer in \cite{zimmer}.
\bigskip

\noindent{\it Acknowledgements}  The recasting of  Fefferman's result  from the point of view of  Mostow rigidity and metric hyperbolicity was the main motivation behind this work, and was outlined  by Michael Cowling, back in 2007. The authors are very grateful to both Michael Cowling and to Loredana Lanzani for several key observations that have led to a better understanding of the problem.

\section{Preliminaries}\label{prelim}
 
In this section we recall some basic definitions and results.
We start by discussing distortion and  conformal maps on subRiemannian manifolds.
Then we discuss pseudoconvex domains and their metrics. Finally we review hyperbolicity in the sense of Gromov.

\subsection{Distorsion in subRiemannian geometry}

By a previous work of the authors together with Ottazzi, we know that several definitions of conformal maps are equivalent in the setting of contact subRiemannian manifolds. We now recall the two definitions that we shall need in this paper.

For a homeomorphism $F:X\to Y$ between  general metric spaces, we
consider the 
  following  quantities 
$$
\Lip_F(x)  := \limsup_{x'\to x} \dfrac{d(F(x), F(x'))}{d(x,x')} \qquad\text{ and }\qquad
\ell_F(x):= \liminf_{x'\to x} \dfrac{d(F(x), F(x'))}{d(x,x')}.$$
The quantity $\Lip_F(x)$ is sometimes denoted by ${\rm Lip}_{ F}(x) $ and is called the pointwise Lipschitz constant. 
Within this paper, we define the {\em distortion} of $f$ at a point $x\in X$ as
\begin{equation}\label{hstella}  H^*(x,F, d_{X}, d_{Y}) := \dfrac{\Lip_F(x) }{\ell_F(x)}.\end{equation}
The homeomorphism $f$ is said to be {\em quasi-conformal} if there exists $K$ such that for all $x\in X$ one has
$$ 
 \limsup_{r \to 0}\frac{
\sup \{d_Y(F(p),F(q)):  {d_X(p,q)\leq r}   \}}
{
\inf \{d_Y(F(p),F(q)) : {d_X(p,q)\geq r}\}}  \leq K   .$$
It is well-known that in the literature there are several other equivalent definitions of quasi-conformality in `geometrically nice' spaces, see \cite{Williams}. 
 However, the equivalence is not quantitative, in the sense that each definition has an associated  constant (like the $K$ above) and the value of of these constants can be different from definition to definition. Thus we need to clarify what is a conformal map.
To do this we invoke 
Theorem 1.3  and Theorem 1.19 from \cite{CLDO}.
Namely, the additional subRiemannian structure allows to an unambiguous definition of 1-quasiconformality.


\begin{lemma}[C-L-O] \label{lemma:CLO}
Let $F:X\to Y$ be a quasi-conformal homeomorphism   between two equiregular subRiemannian manifolds.

(i)  The requirement $  H^*(\cdot,F, d_{X}, d_{Y})\equiv 1$ is equivalent to other notions of 1-quasi-conformality.

(ii) If  $X$ and $Y$ are contact manifolds, then 1-quasi-conformality of $F$ is equivalent to $F$ being conformal (i.e., smooth and with horizontal differential that is a homothety).
\end{lemma}

One of the advantages to work with \eqref{hstella} is that it immediately  yields a  chain rule:
\begin{equation}\label{chain}	
H^*(x,F_1\circ F_2) \leq H^*(x, F_2 )  H^*(F_2 (x),F_1 ).  
\end{equation}
The last equation  follows from the fact that  $ \limsup a_nb_n \leq \limsup a_n \limsup b_n$ whenever $a_n, b_n\geq0$.
 Moreover, we trivially have that if $f$ is an $L$-biLipschitz homeomorphism, then
\begin{equation}\label{biLip}	H^*(x,F ) \leq L.
\end{equation}

\subsection{Pseudoconvex domains and hermitian metrics} 
\label{sec:Pseudoconvex}
We recall some of the basic definitions about pseudoconvex domains and hermitian metrics, as well as some key results proved by Balogh and Bonk in \cite{Balogh-Bonk}.

Let  $D\subset \C^n$, $n\ge 2$ be a smooth, bounded open set.
Let $\defun: \C^n \to \R$ denote the  signed distance function from $\p D$, negative in $D$ and positive in its complement.  
Set $N_\delta=\{x\in D\ |\  d_E(x, \partial D)<\delta\}$.

\begin{lemma}[Tubular Neighborhood Theorem] \label{bb-background}
Let $D\subset \C^n$, $n\ge 2$ be a bounded domain with  smooth boundary.  There exists $\delta_0>0$ such that the projection $\pi:N_{\delta_o}\to \p D$ is a smooth, well defined map and the distance function $d_E(\cdot,\p D)$ is smooth on $N_{\delta_0}$.
\end{lemma}

We will denote by $n(x)$  the outer unit normal at $x\in \p D$, so that the fiber $\pi^{-1}(x)\cap N_{\delta_0}=\{ x+sn(x) | s\in (0,\delta_0)\}.$

For $p\in \p D$,  one can define the tangent space $T_p\p D=\{ Z\in \C^n | Re\langle \bar \p \defun(p), Z\rangle=0  \}$ and its maximal complex subspace $H_p\p D=\{  Z\in \C^n|  \langle \bar \p \defun(p),Z\rangle=0\}$, where $\langle Z,Z'\rangle=\sum_{i=1}^n Z_i \bar Z'_i$ is the  hermitian product. 
  By definition, 
the domain $D$ is {\em strictly pseudoconvex} if for every $p\in \p D$, the Levi form
\begin{equation}\label{levi}
L_\defun(p,Z)  := \sum_{\al,\beta=1}^n \p^2_{z_{\al}\bar z_{\beta}} \defun(p) Z_\al \bar Z_{\beta}
\end{equation}
is positive definite on $H_p\p D$. 

For each $p\in \p D$ one has a splitting $\C^n=H_p\p D\oplus N_p \p D$, where $N_p \p D$ is the complex one-dimensional subspace orthogonal to $H_p \p D$. This splitting at $p$ induces a decomposition $Z=Z_H+Z_N$ for all $Z\in \C^n$.

%
Metrics that are invariant under the action of biholomorphisms play a key role in several complex variables. Important examples are the Bergman metric,    the Kobayashi metric, and the Carath\'eodory metric  (see \cite{MR3114665}). 
In all cases, for $x\in D$ the length of a complex vector $Z\in  T_x D=\C^n$ is given by a Finsler structure $K(x,Z)$.
We will rely on the following result, which can be found in
\cite{MR1679982} and also  \cite[Proposition 1.2]{Balogh-Bonk}.

\begin{proposition}[{{Balogh-Bonk}}]\label{bb-estimate}
Let $D\subset \C^n$, $n\ge 2$ be a bounded, strictly pseudoconvex domain with  smooth boundary and let $K(x,Z)$ 
be the Finsler structure associated to  the  Bergman metric or the Kobayashi metric or the Carath\'eodory metric.
 For every $\bar\e>0$ there exists $\delta_0,C>0$ such that for all $x\in D$ with $d_E(x,\p D)\le \delta_0$ and $Z\in \C^n$ one has 
\begin{multline}\label{bb-1.4}
(1-C \sqrt{d_E(x,\p D) }) \Bigg(  \frac{|Z_N|^2}{4d_E^2(x,\p D)} + (1-\bar\e) \frac{L_\defun(\pi (x), Z_H)}{d_E(x,\p D)} \Bigg)^{\frac{1}{2}} \le K(x,Z) \\ \le
(1 +C  \sqrt{d_E(x,\p D) }) \Bigg( \frac{|Z_N|^2}{4d_E^2(x,\p D)} + (1+\bar\e) \frac{L_\defun(\pi (x), Z_H)}{d_E(x,\p D)}  \Bigg)^{\frac{1}{2}},
\end{multline}
where $Z=Z_H+Z_N$ is the splitting at $\pi(x)$.
\end{proposition}

The subbundle $H \p D$ is a contact distribution on $\p D$ and the triplet $(\p D, H\p D, L_\defun)$ yields a contact subRiemannian manifold. In this structure, the  {\it horizontal} curves are those arcs in $\p D$ that are tangent to the contact distribution, and the {\it Carnot-Carath\'eodory} distance $d_{CC}(p,q)$ between $p,q\in \p D$ is defined as the minimum time it takes to reach one point from the other traveling  along horizontal curves at unit speed with respect to the Levi form, see \cite{Gromov1}.

As in \cite{Balogh-Bonk}, we will need to use a family of Riemannian metrics on $\p D$ that approximate the sub-Riemannian metric associated to the Levi form, and that in fact  have corresponding distance functions that converge in the sense of Gromov-Hausdorff to the Carnot-Carath\'eodory distance. For every $k>0$ we define a Riemannian metric $g_k$ on $T\p D$ as
\begin{equation}\label{app-def}
g_k^2(p, Z):= L_\defun(p,Z_H)+ k^2 |Z_N|^2,
\end{equation}
for every  $p\in \p D$ and every  $Z=Z_H+Z_N\in T_p  \p D$. 
Here we just recall a basic comparison result (see for instance \cite[Lemma 3.2]{Balogh-Bonk}) relating the distance function $d_k$ associated to $g_k$ to the Carnot-Carath\'eodory distance $d_{CC}$.
\begin{lemma}\label{approx}
There exists a constant $C>0$ such that for all $k>0$, and for all points $p,q\in \p D$, with $d_{CC}(p,q)\ge k^{-1}$ one has 
\begin{equation}
C^{-1} d_k(p,q)\le d_{CC}(p,q) \le C d_k (p,q).
\end{equation}
\end{lemma}
\subsection{Gromov Hyperbolicity}

Let $x, y, o $ be three points in a  metric space $(X, d)$.
Then the Gromov product of $x$ and $y$ at $o$, denoted $\langle x,y \rangle_o$, is defined by
\begin{equation*}
\langle x,y \rangle_{o} = \frac1{2} \big( d(x, o) + d(y, o) - d(x, y) \big).
\end{equation*}
Then $X$ is called {\em Gromov hyperbolic} if there exists $\delta\geq 0$ such that
$$\langle x,y \rangle_o\geq \min\{ \langle x, z\rangle_o , \langle z, y\rangle_o\} -\delta, \quad \text{ for all }x, y, z,o \in X.$$

For a Gromov hyperbolic space $X$ one can define a boundary set $\partial_\infty X$ as follows, see \cite[p.431-2]{Bridson:book}.
Fix a basepoint $o\in X$. A sequence $(x_i)$ in $X$ is said to {\em converge at infinity} if
$\lim_{i,j\to\infty}\langle x_i, x_j\rangle_o = \infty$. Two sequences $(x_i)$ and $(y_i)$ converging at infinity are
called equivalent if 
$\lim  \langle x_i, y_i\rangle_o =\infty.$ 
These notions do not depend on the
choice of the basepoint $o$.
 The set $\partial_\infty  X$ is now defined as the set of equivalence
classes of sequences converging at infinity.

 For $p,q \in \partial_\infty  X$ and $o\in  X$ we define 
$$\langle p, q\rangle_o = \sup \liminf_{i\to \infty} \langle x_i, y_i\rangle_o  ,$$
where the supremum is taken over all sequences $(x_i)$ and $(y_i)$ representing the
boundary points $p$ and $q$, respectively.
Actually,   there exists such sequences $(x_i)$ and $(y_i)$ for which 
 $\langle p, q\rangle_o =   \lim _{i\to \infty} \langle x_i, y_i\rangle_o $, see \cite[Remark 3.17]{Bridson:book}.

Balogh and Bonk have proved that if $D\subset \C^n$, $n\ge 2$ is a bounded, strictly pseudoconvex domain with  smooth boundary,
and $K(x,Z)$ is a norm satisfying \eqref{bb-1.4}, then the corresponding metric space $(D,d_K)$ is Gromov hyperbolic and its visual boundary coincides with the topological boundary. See \cite[Theorem 1.4]{Balogh-Bonk}.

%
%

\section{Conformal equivalence of  boundary metrics}
\subsection{Proof of Proposition~\ref{p1} }


In this section we prove  Proposition~\ref{p1}, and then show that the conformal equivalence result holds more in general for every hyperbolic filling.

Let $g$ be as defined in \eqref{filling-metric-cc} and let $\rho^g_o$ be its Bourdon distance, as defined by \eqref{def:Bourdon:dist}. 
We begin by giving a computation of the distance $\rho^g_o$   on two points $p,q\in \p D$. We represent $p$ and $ q$ by two sequences $x_i$ and $y_i\in  D$, respectively.
Notice that since $x_i\to p$ in $\C^n$ then $\pi(x_i)\to p$ and $h(x_i)\to 0$. In particular, we also have that
$\max(h(x_i),h(o)) = h(o)$. Similar considerations apply to $y_i$ and $q$.
We compute
 \begin{eqnarray*} 
 \rho^{g}_o ( p,q) &=& \exp (- \langle p,q \rangle_o) \\
 &=& \lim_{i\to \infty} \exp (- \langle x_i, y_i \rangle_o) \\
  &=& \lim_{i\to \infty} \exp (- \half( g(x_i,o) + g(y_i,o) - g(x_i,y_i ))) \\
    &=& \lim_{i\to \infty} \exp \bigg(
    -
     \log (   \frac{d_{\rm CC}(\pi(x_i), \pi(o))+ \max(h(x_i),h(o) ) }{\sqrt{h(x_i) h(o)} } )\\
     & & \hspace{3cm}
    -
    \log (   \frac{d_{\rm CC}(\pi(y_i), \pi(o))+ \max(h(y_i),h(o) ) }{\sqrt{h(y_i) h(o)} } )
    \\
     & & \hspace{5cm}
    +
    \log (   \frac{d_{\rm CC}(\pi(x_i), \pi(y_i))+ \max(h(x_i),h(y_i) ) }{\sqrt{h(x_i) h(y_i)} } )  \bigg) \\
        &=&
         \lim_{i\to \infty}  
   \bigg(   \frac{d_{\rm CC} (p, \pi(o))+ h(o)  }{\sqrt{h(x_i) h(o)} } \bigg)^{-1}
     \bigg(   \frac{d_{\rm CC}(q, \pi(o))+ h(o)  }{\sqrt{h(y_i) h(o)} } \bigg)^{-1}
        \frac{d_{\rm CC}(p,q)   }{\sqrt{h(x_i) h(y_i)} }    \\
                &=& 
  \frac{
 d_{\rm CC}(p,q) \,\,   h(o)    }{
   (d_{\rm CC}(p, \pi(o))+ h(o)  )(   d_{\rm CC}(q, \pi(o))+ h(o)  )}
        .
 \end{eqnarray*}
For every $p\in \p D$ one has 
$$\lim_{q\to p} \dfrac{ \rho^{g}_o ( p,q) }{d_{CC}(p,q)}
=
\lim_{q\to p} 
  \frac{
   h(o)    }{
   (d_{\rm CC}(p, \pi(o))+ h(o)  )(   d_{\rm CC}(q, \pi(o))+ h(o)  )}
   = 
  \frac{
   h(o)    }{
   (d_{\rm CC}(p, \pi(o))+ h(o)  )^2}
$$ 
so the limit exists, and  the identity map $ (\p D, d_{CC}) \to (\p D,\rho^g_o)$ is 1-quasi-conformal.
\qed

\subsection{Boundary distances of hyperbolic fillings}\label{bbs} An important contribution of  Bonk and Schramm \cite{MR1771428}, is that the functor $X\to \p_\infty X$ has an inverse functor, in the form of {\it hyperbolic filling} spaces $Con(Z)$. To be more precise,
 one defines $Con(Z) = Z \times (0,D)$, endowed with the metric  given by 
\begin{equation}\label{filling-metric}
d_2((x,u),(y,v))=2\log \bigg(   \frac{d_1(x, y)+ \max(u,v ) }{\sqrt{u v} } \bigg).
\end{equation}
The space $(Con(Z),d_2)$ is Gromov hyperbolic, and its visual boundary is $Z$, with the canonical class of snowflake equivalent  metrics   given by $d_1$. Here we note that a particular visual metric is actually conformal to $d_1$.
 We will consider the particular  visual metric generated by $g$ given by the Bourdon distance. Choose a generic base point choose a base point $o=(z,s)$, with  $z\in Z$ and $s\in (0,D)$. For any  two points $x,y\in  Z$ so that $d_1(x,y) < s$.
consider  $u,v \in (0, d_1(x,y) )$. Following \eqref{def:Bourdon:dist}, the Bourdon distance $d_2(x,y)$ is defined as follows
$$d_2(x,y)=\lim_{u,v\to 0} e^{-\langle ( x ,u)  ,  (y , v)  \rangle_o}.$$

Notice that in general,  
  Bourdon distances 
  associated to the hyperbolic fillings are a quasi-distance.
  By quasi-distance we intend   that  the triangle inequality is satisfied modulo a multiplicative constant.

\begin{proposition}
Let  $d_1$ a distance on a bounded space $Z$. If
$d_2$ denotes the 
  Bourdon distance associated to the hyperbolic filling  for $d_1 $,
  then $d_1 $ and $d_2$ are conformally equivalent.
  \end{proposition}
  \proof
 In order to show that  $d_1, d_2$ are conformally equivalent it suffices to prove that
 the limit $ \lim_{y\to x}d_1 (x,y) / d_2(x,y) $ exists for every $x\in Z$.
Fix any $z\in Z$ and $s\in (0,D)$. Let $o=(z,s)$.
Take two points $x,y\in  Z$ so that $d_1(x,y) < s$.
Take $u,v \in (0, d_1(x,y) )$.

The rest of the proof follows from
$$ 
\langle ( x ,u)  ,  (y , v)  \rangle_o
=\half( d_2(( x ,u),o) + d_2((y , v),o) - d_2(( x ,u),(y , v)) )$$
$$
= \log   \bigg(   \frac{d_1(x, z)+ \max(u,s ) }{\sqrt{u s} } \bigg)
+\log \bigg(   \frac{d_1(y, z)+ \max(v,s ) }{\sqrt{v s} } \bigg)
-\log \bigg(   \frac{d_1(x,y)+ \max(u,v) }{\sqrt{uv} } \bigg)
  $$
  $$
= \log  \frac{ (   d_1(x, z)+ s ) 
(    d_1(y, z)+ s  )}
 { s  (    d_1(x,y)+ \max(u,v)  )}
  $$
    $$
= \log \bigg( \frac{ (   d_1(x, z)+ s ) 
(    d_1(y, z)+ s  )}
 { s  (    d_1(x,y)+ \max(u,v)  )}
 \frac{ d_1(x,y) } { d_1(x,y)}  \bigg)
  $$
      $$
= - \log \bigg( d_1(x,y)  
    \bigg) +
   \log \bigg( 
 \frac{ d_1(x,y) (   d_1(x, z)+ s ) 
(    d_1(y, z)+ s  )}
  { s  (    d_1(x,y)+ \max(u,v)  )}
 \bigg)
  $$
  
  We calculate $ \lim_{y\to x}d_1 (x,y) / d_2(x,y) $. Consider the quotient
\begin{eqnarray*}
d_2 (x,y) / d_1(x,y)  &=& \lim_{u,v\to 0} \frac{e^{-\langle ( x ,u)  ,  (y , v)  \rangle_o }}{ d_1(x,y)}
\\
&=&\lim_{u,v\to 0} \frac{ e^{ \log \big( d_1(x,y)  
    \big) } e^{- \log \big( 
 \dfrac{ d_1(x,y) (   d_1(x, z)+ s ) 
(    d_1(y, z)+ s  )}
  { s  (    d_1(x,y)+ \max(u,v)  )}
 \big)  } } { d_1(x,y)}
 \\
 &=&
 \lim_{u,v\to 0} \bigg( 
 \frac{ d_1(x,y) (   d_1(x, z)+ s ) 
(    d_1(y, z)+ s  )}
  { s  (    d_1(x,y)+ \max(u,v)  )}
 \bigg)^{-1}\\
 &=&\frac{ s     d_1(x,y)  }{ d_1(x,y) (   d_1(x, z)+ s ) 
(    d_1(y, z)+ s  )  }
\\
&=& \frac{s}{ (   d_1(x, z)+ s ) 
(    d_1(y, z)+ s  )}.
\end{eqnarray*}

The latter implies that 
$$ \lim_{y\to x}\frac{d_2 (x,y) }{ d_1(x,y)}=  \frac{s}{ (   d_1(x, z)+ s )^2} ,$$
which gives the conclusion.
\qed

\section{Comparing $d$ and $g$, after Balogh and Bonk}

The  quantitative bounds on the distortion of the identity map in Proposition~\ref{p2} follow from the following result, which is a refinement of an analogue statement of Balogh and Bonk \cite[Corollary 1.3]{Balogh-Bonk}.  We follow largely their arguments, but where in \cite{Balogh-Bonk} the noise due to the rough geometry would yield an additive  constant, here instead we need to exploit the fact that the geometry is asymptotically hyperbolic to show that such constants can be chosen arbitrarily small the closer one gets to the boundary.

   \begin{theorem}\label{bb-better}
  For every $\bar p\in  \p D$ and $\e>0$ there exists
  $r>0$
 such that
    for all distinct $p,q\in \p D\cap B(\bar p,r)$  
    there exists
    $r'>0$ 
such that
    for all $ x \in  D\cap B(p,r') $ and all $y \in   D\cap B(q,r') $
 \begin{equation}\label{let-me-tell-ya}
 g(x, y) -\e \le d_K(x, y)\le g(x, y) +\e .
 \end{equation}
  \end{theorem}
 In the rest of the paper we will refer to this result in connection with the quintuplet $(\bar p, p,q,x,y)$.
   
%

 \subsection {Lemmata}
  The proof of Theorem~\ref{bb-better} is based on preliminary estimates established in Lemma~\ref{bb-one is a base point},
  Lemma~\ref{bb-both are arbitrarily close2}, 
  and  Lemma~\ref{curve goes far from boundary} below.
  %
%
%

\begin{lemma}\label{bb-one is a base point}
Let $\delta_0>0$ to be the constant in Lemma~\ref{bb-background}.
For $x_1,x_2\in D$ with $d_E(x_i, \p D)<\delta_0$, and $h(x_1)\ge h(x_2)$, consider a piecewise $C^1$ curve $\gamma:[0,1]\to N_{\delta_0}$ with $\gamma(0)=x_1$ and $\gamma(1)=x_2$. The length $l_K(\gamma)$ of $\gamma$ with respect to the metric $d_K$
 satisfies
\begin{equation}\label{bound-below}
l_K(\gamma)\ge \ln \frac{h(x_1)}{h(x_2)} - C\bigg(h(x_1)- h(x_2)\bigg),
\end{equation}
where $C$ is the same constant as in \eqref{bb-1.4}.
Moreover, if  the curve is a segment $\gamma(t)=x_1+t(x_2-p_1)\subset \pi^{-1}(p)$ for some $p\in \p D$ then one has
\begin{equation}\label{bound-above}
l_K(\gamma)\le \ln \frac{h(x_1)}{h(x_2)} + C\bigg(h(x_1)- h(x_2)\bigg),
\end{equation}
\end{lemma}

\begin{proof}
Recall from \cite[page 517]{Balogh-Bonk} that 
$$\bigg| \frac{d}{dt} h(\gamma(t)) \bigg| =\frac { | Re\langle n(\pi(\gamma(t))), \gamma'(t)\rangle |}{2h(\gamma(t))} \le \frac{|\gamma_N'(t)|}{2 h(\gamma(t))}.$$
The latter and \eqref{bb-1.4}  yield
\begin{eqnarray*}
l_K(\gamma)&=& \int_0^1 K(\gamma(t), \gamma'(t)) dt \\
& \ge&
\int_0^1 (1- C d_E^{\frac{1}{2}}(\gamma(t), \p D)) \bigg( 
\frac{ |\gamma_N'(t)|^2}{4 d_E^2(\gamma(t), \p D)} + (1-\bar\e) \frac{ L_\defun(\pi(\gamma(t), \gamma_H'(t)) }{ d_E(\gamma(t), \p D)}  \bigg)^\frac{1}{2} dt 
\\ &\ge& \int_0^1 (1- C d_E^{\frac{1}{2}}(\gamma(t), \p D))
\frac{ |\gamma_N'(t)|}{2 d_E(\gamma(t), \p D)}dt  \ge  \int_0^1 \frac{(1- C h(\gamma(t))}{h(\gamma(t))} \frac{d}{dt} h(\gamma(t)) dt
\\ &=& \ln \frac{h(x_1)}{h(x_2)} - C (h(x_1)-h(x_2)),
\end{eqnarray*}
which gives \eqref{bound-below}.

On the other hand, if $\gamma(t)=x_1+t(x_2-x_1)$, then we observe that $\gamma'$ is parallel to the unit normal at $\pi(x_i)$ and so has no tangent component, hence no horizontal component with respect to the splitting at $\pi(x_i)$.  Using the fact that $$ d_E(\gamma(t),\p D)= |\gamma(t)-p|=| t (x_2-x_1)|+|x_1-p|,$$  and \eqref{bb-1.4}  one has
\begin{eqnarray*}
l_K(\gamma)&= &\int_0^1 K(\gamma(t), \gamma'(t)) dt \\
& \le&
\int_0^1 (1+ C d_E^{\frac{1}{2}}(\gamma(t), \p D)) \bigg( 
\frac{ |\gamma_N'(t)|^2}{4 d_E^2(\gamma(t), \p D)} + (1+\bar\e) \frac{ L_\defun(\pi(\gamma(t), \gamma_H'(t)) }{ d_E(\gamma(t), \p D)}  \bigg)^\half dt 
\\ &=& \frac{1}{2} \int_0^1 \frac{ |x_2-x_1|}{ t |x_2-x_1|+ |x_1-p|} dt + \frac{C}{2}  \int_0^1  \frac{1}{\sqrt{  t |x_2-x_1|+ |x_1-p| }} dt \\
& = &\ln \frac{h(x_1)}{h(x_2)} + C (h(x_1)-h(x_2)),
\end{eqnarray*}
which gives \eqref{bound-above}.
\end{proof}

An immediate consequence of Lemma \ref{bb-one is a base point} is the following.
\begin{corollary}\label{sulla stessa fibra}  Let $\delta_0>0$ to be the constant in Proposition~\ref{bb-background}.
If  $x_1,x_2\in D$, with  $\delta_0>h(x_1)\ge h(x_2)$, then 
\begin{equation*}
\ln \frac{h(x_1)}{h(x_2)} - C\big(h(x_1)- h(x_2)\big) \le d_K(x_1,x_2).\end{equation*}
Moreover, if  $\pi(x_1)=\pi(p_2)$,
then we also have
\begin{equation}\label{quattrodieci}
 d_K(x_1,x_2) \le \ln \frac{h(x_1)}{h(x_2)} + C\big(h(x_1)- h(x_2)\big)
\end{equation}
where $C$ is the same constant as in \eqref{bb-1.4}.
\end{corollary}
\bigskip

\bigskip

The next lemma provides an upper bound for $d_K(x_1,x_2)$  in the case when both points $x_1,x_2$ are at the same distance from the boundary and equal to the Carnot-Carath\'eodory distance between the projections $\pi(x_1), \pi(x_2)$.

\begin{lemma}\label{bb-both are arbitrarily close2}
Let $p_1,p_2\in \p D$. If we  set $ x_i := p_i - d_{CC}(p_1,p_2) n(p_i)  $, $i=1,2$,
then 
\begin{equation}
d_K(x_1,x_2)\to 0, \qquad \text{as } \,d_{CC}(p_1,p_2)\to 0.
\end{equation}  
%
\end{lemma}

\begin{proof}
Let $\eta>0$ and let  $\alpha:[0,1]\to \p D$ be any horizontal curve with $\alpha(0)=p_1$ and $\alpha (1)=p_2$, such that its subRiemannian length $l_{CC}$, satisfies
 $$l_{CC}(\alpha)= \int_0^1 L^{\frac{1}{2}}_\defun(\alpha(t), \alpha'(t)) dt \le d_{CC}(p_1,p_2)(1+\eta).$$ Define a new curve $\gamma:[0,1]\to D$ as a {\it lift} at height $h\in (0,\delta_0)$ of $\alpha$ by the formula
\begin{equation}\label{lift-up}
\gamma(t):= \alpha(t) - h \  n(\alpha(t)).
\end{equation}
Arguing as in the proof of  \cite[Lemma 2.2]{Balogh-Bonk}  yields the following relations between $\alpha' $ and $\gamma'$,
\begin{eqnarray}\label{lift-est}
 L(\alpha(t), \gamma_H'(t)) &=&L(\alpha(t), \alpha'(t))  +O(h |\alpha'(t)|^2 ) \\
|\gamma_N'(t)| &=& O(h |\alpha'(t)|). \notag
\end{eqnarray}
In fact, from \eqref{lift-up} one has $\gamma'(t)|_H= \alpha'(t)-[dn|_{\alpha(t)} \alpha'(t)]_H$ which,  together with the bilinearity of the Levi form, yields \eqref{lift-est}.
Consequently we have
\begin{eqnarray*}
l_K(\gamma)&=& \int_0^1 K(\gamma(t), \gamma'(t)) dt \\
& \le&
\int_0^1 (1+ C d_E^{\frac{1}{2}}(\gamma(t), \p D)) \bigg( 
\frac{ |\gamma_N'(t)|^2}{4 d_E^2(\gamma(t), \p D)} + (1+\bar\e) \frac{ L_\defun(\pi(\gamma(t), \gamma_H'(t)) }{ d_E(\gamma(t), \p D)}  \bigg)^\half dt 
\\
&=& \int_0^1 (1 +C h^{\half} ) \bigg( 
\frac{ |\gamma_N'(t)|^2}{4 h^2} + (1+\bar\e) \frac{ L_\defun(\pi(\gamma(t), \gamma_H'(t)) }{ h}  \bigg)^\half dt 
\\ &\le& \int_0^1 (1 +C h^{\half} ) \bigg(  C |\alpha'(t)|^2 + (1+\bar\e) \frac{L_\defun(\alpha(t), \alpha'(t))}{h}
 \bigg)^\half dt 
 \\ &\le&
 (1+C\sqrt{h}) (1+\eta) \bigg[C d_{CC}(p_1,p_2)  + (1+\bar\e)^\half h^{-\half} d_{CC}(p_1,p_2)\bigg]
 .
\end{eqnarray*}
Setting $h=d_{CC}(x_1,x_2)$ in the latter yields the conclusion.
\end{proof}

The next lemma will be instrumental  in establishing a lower bound for $d_K(x_1,x_2)$  in the case when a length minimizing arc $\gamma$ joining two   points   $x_1,x_2\in D$ will travel at a distance further than the Carnot-Carath\'eodory distance between their projections.

\begin{lemma}\label{curve goes far from boundary}
Let $\delta_0>0$ be smaller than the similarly named constants in Propositions~\ref{bb-background} and \ref{bb-estimate}.
Consider two points $x_1,x_2\in D$ with $d_E(x_i,\p D)  <\delta_0$.
Set  $p_i=\pi (x_i)\in \p D$, and let $\gamma:[0,1]\to D$ denote an arc joining $x_1$ and $x_2$.
If $ \max_{z\in \gamma} h(z)\ge d_{CC}(p_1,p_2)$  then
\begin{equation}\label{bound-above-h}
l_K(\gamma) \ge 
2 \ln \bigg( \frac{d_{CC}(p_1,p_2)}{\sqrt{h(x_1)h(x_2)} } \bigg)- C(2d_{CC}(p_1,p_2)- h(x_1)-h(x_2)),
\end{equation} 
where $C$ is the same constant as in \eqref{bb-1.4}. 
\end{lemma}
\begin{proof} Choose $t_0\in [0,1]$ such that $h(\gamma(t_0))=\max_{z\in \gamma} h(z)$. Set $\gamma_1,\gamma_2$ be the two branches of the curve $\gamma$ corresponding to the subintervals $[0,t_0]$ and $[t_0,1]$.
Set also $\bar \gamma_1$ and $\bar \gamma_2$ to be the connected components on $\gamma_1$ and $\gamma_2$ joining $x_i$ to the closest points $y_i\in \gamma$ such that $h(y_i)=d_{CC}(p_1,p_2)$, for $i=1,2$. More formally, $y_1=\gamma(t_1)$, with $t_1=\inf\{t\in [0,t_1]  \ $such that $h(\gamma(t))\ge d_{CC}(p_1,p_2)\}$. The point $y_2$ is defined analogously.

Next we invoke Lemma~\ref{bb-one is a base point} to deduce
\begin{eqnarray*}
l_K(\gamma)&\ge& l_K(\bar \gamma_1)+ l_K(\bar \gamma_2) \\
\\ &\ge& \ln \frac{d_{CC}(p_1,p_2)}{h(x_1)} +  \ln \frac{d_{CC}(p_1,p_2)}{h(x_2)} - C(2d_{CC}(p_1,p_2)- h(x_1)-h(x_2)) \\
&= &2 \ln \bigg( \frac{d_{CC}(p_1,p_2)}{\sqrt{h(x_1)h(x_2)} } \bigg)- C(2d_{CC}(x_1,x_2)- h(x_1)-h(x_2)),
\end{eqnarray*}
which is the desired bound \eqref{bound-above-h}.
\end{proof}
\bigskip

\subsection{Proof of Theorem~\ref{bb-better}}
Thanks to the previous lemmata we can now prove the main result of the section.

\begin{proof}[Proof of Theorem~\ref{bb-better}]
We shall show that  for  all $\bar p\in  \p D$ and $\e>0$ one can choose 
  $r>0$ small enough so that 
    for all distinct $p,q\in \p D\cap B(\bar p,r)$  
one can find
    $r'\in(0, r)$ 
such that
    \eqref{let-me-tell-ya} holds    for all $ x \in  D\cap B(p,r') $ and all $y \in   D\cap B(q,r') $.
    In our proof we begin with arbitrary values of $r$ and $r'$ and then put several constrains on them.

 If $p$ and $q$ are distinct, then the value $d_1:=d_{CC}(p,q)$ is strictly positive.
 We shall choose $r$ smaller that the constants $\delta_0$ in Propositions~\ref{bb-background} and \ref{bb-estimate} and so that  $d_1 $ is small enough to be determined later. 
 Denote by $\bar x$, and $\bar y$ the projections on the boundary of $x$ and $y$, respectively. Note that since the projections are the closest points in $\p D$, then $\bar x\in B(p,2r')$ and $\bar y\in B(q,2r')$.
 Set $d_2:=d_{CC}(\bar x, \bar y)$.
 Notice that as $r'\to 0$ we have $d_2\to d_1$.
 We shall choose $r' $ sufficiently small  so that $r' <d_2$ and
$d_2\in (d_1/2, 2 d_1)$. In particular, 
if $r$ was chosen small enough, then $d_2$ is positive and smaller than the constants $\delta_0$ in Propositions~\ref{bb-background} and \ref{bb-estimate}. 

\noindent {\it Proof of the upper bound in   \eqref{let-me-tell-ya}.}  Set $x':= \bar x - d_2  n(\bar x)$ and $y':= \bar y - d_2  n(\bar y)$, so $x'$, $y'$ are points  in $D$ at   distance $d_2$ from $\p D$ and with the same projection on $\p D$ as $x$, $y$, respectively.
 
By Lemma~\ref{bb-both are arbitrarily close2} we can choose   $d_1$ sufficiently   small   so that
$d_K(x', y')<\eps/3$. 
 Invoking \eqref{quattrodieci}, since $h(x')=h(y')=d_2>\max\{h(x), h(y)\}$, yields $$d_K(x,x') \leq \log(d_2/h(x)) + C( d_2 - h(x) )\qquad \text{and }\qquad 
 d_K(y,y') \leq \log(d_2/h(y)) + C( d_2 - h(y) )
 .$$
 Choose $d_1$   chosen sufficiently small so that $C d_2 \leq \eps/3$.
  
 
Combining the previous bounds with the definition of $g$, we obtain the following estimates
\begin{eqnarray*}
d_K(x,y)-g(x,y)&\le& d_K(x,x')+d_K(x',y')+d_K(y',y)-g(x,y)
\\
&\le&   \log(d_2/h(x)) + C( d_2 - h(x) ) +
\eps/3+
\log(d_2/h(y)) + C( d_2 - h(y) )
   \\& & \hspace{6.6cm}- 2\ln \bigg( \frac{d_2 +   h(x) \wedge h(y))}{\sqrt{h(x) h(y)}}\bigg)\\
&\le& \eps -C h(x) - C h(y) - 2 \ln \bigg(   1 +   \frac{ h(x) \wedge h(y)}{d_2}  \bigg) < \eps,
 \end{eqnarray*}
where we used that the terms
$h(x),h(y) , \ln (   1 +   \frac{ h(x) \wedge h(y)}{d_2}  )  $ are positive. This conclude the proof of the upper bound in    \eqref{let-me-tell-ya}.
 
 \bigskip

\noindent {\it Proof of the lower bound in   \eqref{let-me-tell-ya}.}
Choose $\delta>0$ such that $\ln (1/(1+\delta))<\e$ and  
$r'>0$ small enough so that $\frac{\max(h(x),h(y))}{d_{CC}(p,q)} \le \delta$,  for all $ x \in  D\cap B(p,r') $ and all $y \in   D\cap B(q,r') $.
Consider any arc $\gamma:[0,1]\to D$  joining $x$ and $y$, and set $H:=\max_{z\in \gamma} h(z)$.

\medskip

\noindent{\it - If $H \ge d_{CC}(\bar x,\bar y)$ } then  in view of Lemma~\ref{curve goes far from boundary} we have
\begin{eqnarray}\nonumber
d_K(x,y)- g(x,y) &\geq& 2 \ln \big( \frac{d_2}{\sqrt{h(x)h(y)} } \big)- C(2d_2- h(x)-h(y)))- 2\ln \big( \frac{d_2 +   h(x) \wedge h(y))}{\sqrt{h(x) h(y)}}  \big)\\
\label{a-z}
&\ge& 2 \ln \bigg( \frac{1}{1+ \frac{\max(h(x),h(y))}{d_2}} \bigg) - C(2d_2-h(x)-h(y)) \\
 & \ge& - (C+2) \e.
 \nonumber
\end{eqnarray}
In this case the proof is concluded.

\medskip

\noindent{\it - If $H \le d_{CC}(\bar x,\bar y)$ }  then it follows that $H$ is smaller than the constants $\delta_0$ in Propositions~\ref{bb-background} and \ref{bb-estimate}. In particular we can assume without loss of generality that $CH<1/2$, where $C$ is as in \eqref{bb-1.4}.
Let $t_0\in [0,1]$ be such that $h(\gamma(t_0))=H$ and consider the two branches $\gamma_1,\gamma_2$ of $\gamma$ given by 
restrictions to $[0,t_0]$ and $[t_0,1]$. Given $\e>0$ as in the statement, let $\theta\in (1,2]$ so that $\ln \theta<\e$ and define $k\in \N$ such that $$h(\gamma(0))\in \Bigg[\frac{H}{\theta^k},   \frac{H}{\theta^{k-1}}\Bigg].$$
Following \cite{Balogh-Bonk}, we define $s_0,s_1,...,s_k\in [0,t_0]$ such that $s_0=0$ and $$s_l = \min \Bigg\{s\in [0,t_0] \ \text{ such that } h(\gamma(s))=  \frac{H}{\theta^{k-l}}\Bigg\}.$$
Set $t_1=s_k\le t_0$ and for each $l=1,...,k$,
$$\nu_l^{-1}=   \frac{d_{CC}(\bar x, \bar y) \cdot(\theta-1) }{ 8\theta^{k-l}}.$$
For each of the two branches $\gamma_1,\gamma_2$, we distinguish two alternatives:

\begin{itemize}
\item {\bf Alternative $\# 1$ (All sub-arcs have large slope)}  In this alternative we assume that for every $l=1,...,k$ one has 
\begin{equation}\label{alternative 1}
d_{CC}(\pi(\gamma(s_{l-1})), \pi(\gamma(s_l))) \le \nu_l^{-1} 
\end{equation}

\medskip
From the latter  we draw two conclusions. The first is a simple application of the triangle inequality,
\begin{multline}\label{A2-1}
d_{CC} (\bar z, \pi(\gamma(t_1)) \\
\le \sum_{l=1}^k d_{CC} (\pi(\gamma(s_{l-1})), \pi (\gamma(s_l))) 
\le   (\theta-1) \frac{d_{CC}(\bar x,\bar y)} {8 \theta^k}\sum_{l=1}^k \theta^l 
\le \frac{d_{CC}(\bar x,\bar y)}{4}.\tag{A1 (i)}
\end{multline}
  
  On the other hand, in view of Lemma~\ref{bb-one is a base point} one has
  \begin{equation}\label{A2-2}
  l_K(\gamma|_{[0,t_1]}) \ge \ln \frac{h(\gamma(t_1))}{h(x)} -C(h(\gamma(t_1))-h(x)) = \ln \frac{H}{h(x)} -C(H-h(x)) .\tag{A1 (ii)}
  \end{equation}
  
  \bigskip
  
\item  {\bf Alternative $\# 2$   (One sub-arc has small slope)}  In this alternative, we assume that there exists  $l\in\{1,...,k\}$ such that
\begin{equation}\label{alternative 2}
d_{CC}(\pi(\gamma(s_{l-1})), \pi(\gamma(s_l))) > \nu_l^{-1} 
\end{equation}

Note that if $s\in [s_{l-1},s_l]$ then from the definition of the points $s_l$, one has
$$h(\gamma(s)) \le \theta^{l-k}H\le \frac{8}{\theta-1} \nu_l^{-1}.$$
We then claim that there exists 
a constant $\mathcal C>0$ depending only on the defining function $\defun$ such that
\begin{equation}\label{approx-and-shout} 
  l_K(\gamma|_{[s_{l-1},s_l]}) 
\mathcal C (\theta-1)^2 \frac{d_{CC}(\bar x,\bar y)}{H}.
\end{equation}
Indeed,
arguing as in \cite[page 521]{Balogh-Bonk} we invoke \eqref{bb-1.4} and Lemma \ref{approx}
and we bound as follows:
\begin{eqnarray*}
\nonumber & & \hspace{-1cm} l_K(\gamma|_{[s_{l-1},s_l]})  \\
 & \ge& \mathcal C \frac{(1-CH)\theta^{k-l}}{H} 
\int_{s_{l-1}}^{s_l} \bigg[
L_{\defun}(\pi(\gamma(s)), [\pi(\gamma(s))]'_H) 
+(\theta-1)^2 \nu_l^{2} |[\pi(\gamma(s))]'_N|^2  \bigg]^{\frac{1}{2} }ds \\
\nonumber &\ge& \frac{\mathcal C}{2}(\theta-1) \frac{\theta^{k-l}}{H} 
\int_{s_{l-1}}^{s_l} \bigg[
L_{\defun}(\pi(\gamma(s)), [\pi(\gamma(s))]'_H) 
+ \nu_l^{2} |[\pi(\gamma(s))]'_N|^2  \bigg]^{\frac{1}{2} }ds
\\ \nonumber &\ge& \mathcal C (\theta-1) \frac{\theta^{k-l}}{H} d_{\nu_l} (\pi(\gamma(s_{l-1})), \pi(\gamma(s_l)))
\\ \nonumber &\ge& \mathcal C (\theta-1) \frac{\theta^{k-l}}{H} d_{CC} (\pi(\gamma(s_{l-1})), \pi(\gamma(s_l))) \\
\nonumber&\ge&
\mathcal C (\theta-1)^2 \frac{d_{CC}(\bar x,\bar y)}{H},
\end{eqnarray*}
where $ d_{\nu_l}$ denotes the approximation of the Carnot-Caratheodory metric defined in
\eqref{app-def}. 

Next we claim that 
\begin{equation}\label{A2}
l_L(\gamma|_{[0,t_1]})
\ge   \ln \bigg(   \frac{H}{h(y)} \bigg) + \frac{\mathcal C (\theta-1)^2}{H} d_{CC}(\bar x,\bar y)- 
C\big(H -h(y)\big) -\e.\tag{A2}
\end{equation}
Indeed, Lemma~\ref{bb-one is a base point} and \eqref{approx-and-shout} 
 yields  
\begin{eqnarray*} 
l_L(\gamma|_{[0,t_1]})&=& l_K(\gamma|_{[0,s_{l-1}]}) + l_K(\gamma|_{[s_{l-1},s_l]}) + l_K(\gamma|_{[s_l,t_1]} ) \\
&\ge& \ln \bigg(  \frac{H}{h(\gamma(s_l))} \frac{h(\gamma(s_{l-1}))}{h(x)} \bigg) + \frac{\mathcal C (\theta-1)^2}{H} d_{CC}(\bar x,\bar y)
\\ && \hspace{6cm}
-C\big(H-h(\gamma(s_{l}) + h (\gamma(s_{l-1})) -h(x)\big)
 \\
&\ge & \ln \bigg(   \frac{H}{h(x)} \theta^{-1} \bigg) + \frac{\mathcal C (\theta-1)^2}{H} d_{CC}(\bar x,\bar y)- 
C\big(H -h(x)\big) \\
&\ge &  \ln \bigg(   \frac{H}{h(x)} \bigg) + \frac{\mathcal C (\theta-1)^2}{H} d_{CC}(\bar x,\bar y)- 
C\big(H -h(x)\big)-\e. 
\end{eqnarray*}
\end{itemize}

\bigskip
Applying similar consideration to the branch $\gamma_2$ one obtains a $t_2\in [t_0,1]$ such that one of the following two alternatives hold: Either

\begin{equation}\label{B1}
d_{CC} (\bar y, \pi(\gamma(t_2)) \\
\le  \frac{d_{CC}(\bar x,\bar y)}{4} \ \text{ and } \ \ 
  l_K(\gamma|_{[t_2,1]}) \ge \ln \frac{H}{h(y)} -C(H-h(y)) . \tag{B1}
  \end{equation}
or

\begin{equation}\label{B2}
l_L(\gamma|_{[t_2,1]})
\ge   \ln \bigg(   \frac{H}{h(y)} \bigg) + \frac{\mathcal C (\theta-1)^2}{H} d_{CC}(\bar x,\bar y)- 
C\big(H -h(y)\big) -\e.\tag{B2}
\end{equation}
 
 To conclude the proof we need to examine all possible combinations of these alternatives.
 We will show that in each case one obtains
  \begin{equation}\label{final-fornow} l_K(\gamma)\ge 
 2 \ln \bigg(\frac{d_{CC}(\bar x,\bar y)}{\sqrt{h(x)h(y)}}  \bigg)  - C( 2d_{CC}(\bar x, \bar y) -h(x)-h(y)) -\e .\end{equation}

 \begin{itemize}
 
 \item {\bf Suppose both (A1) and (B1) hold.} Observe that 
 \begin{eqnarray*} 
 d_{CC}(\pi(\gamma(t_1)),\pi(\gamma(t_2))) &\ge& d_{CC}(\bar x,\bar y) - d_{CC}(\bar x,\pi(\gamma(t_1))) - d_{CC}(\bar y,\pi(\gamma(t_2))) \\ &\ge& \frac{ d_{CC}(\bar x,\bar y)}{2}. \end{eqnarray*}
 Repeating the argument in \eqref{approx-and-shout} for $l=k$ and invoking the Riemannian approximation lemma \cite[Lemma 2.2]{Balogh-Bonk} one has
 $$l_L(\gamma|_{[t_1,t_2]}) \ge \mathcal C (\theta-1) \frac{d_{\nu_k}(\pi(\gamma(t_1)),\pi(\gamma(t_2)))}{H}  \ge \mathcal C (\theta-1) \frac{d_{CC}(\bar x,\bar y)}{H}.$$
 The latter, together with \eqref{A2-2}, and the second inequality in \eqref{B1} yields
 $$l_K(\gamma) \ge 2 \ln \bigg(\frac{H}{\sqrt{h(x)h(y)}}  \bigg) + \mathcal C (\theta-1)  \frac{d_{CC}(\bar x,\bar y)}{H} - C( 2H-h(x)-h(y)).$$
 Since the right hand side is monotone decreasing in $H\le d_{CC}(\bar x,\bar y)$ then one has
 \begin{eqnarray*}
  l_K(\gamma)&\ge& 
 2 \ln \bigg(\frac{d_{CC}(\bar x,\bar y)}{\sqrt{h(x)h(y)}}  \bigg) + \mathcal C (\theta-1) - C( 2d_{CC}(\bar x,\bar y) -h(x)-h(y)) \\
& \ge& 2 \ln \bigg(\frac{d_{CC}(\bar x,\bar y)}{\sqrt{h(x)h(y)}}  \bigg)  - C( 2d_{CC}(\bar x,\bar y) -h(x)-h(y))  \end{eqnarray*}
 completing the proof of \eqref{final-fornow}.
  \item {\bf Suppose both (A1) and (B2) hold.} One immediately has
  \begin{multline*} l_K(\gamma)\ge l_K(\gamma_{[0,t_1]}) +  l_K(\gamma|_{[t_2,1]})
  \\
  \ge   \ln \bigg(   \frac{H}{h(x)} \bigg) + \frac{\mathcal C (\theta-1)}{H} d_{CC}(\bar x,\bar y)- 
C[H -h(x)]-\e +
    \ln \frac{H}{h(y)} -C(H-h(y)).
  \end{multline*}
  Applying the same consideration as above we immediately deduce \eqref{final-fornow}.
   \item {\bf Suppose both (A2) and (B1) hold.} This combination  is dealt with analogously to the previous case.
    \item {\bf Suppose both (A2) and (B2) hold.} Estimate \eqref{final-fornow} follows immediately from \eqref{A2} and \eqref{B2}.
 \end{itemize}

To conclude the proof we need to consider the infimum of $l_K(\gamma)$ among all arcs $\gamma$ joining $x$ and $y$ and apply \eqref{final-fornow} to each. One has
\begin{eqnarray*}
d_K(x,y)-g(x,y)& \ge&
 2 \ln \bigg(\frac{d_{CC}(\bar x,\bar y)}{\sqrt{h(x)h(y)} }
 \frac{1}{ \frac{d_{CC}(\bar x,\bar y)}{\sqrt{h(x)h(y)}} +\frac{\max\{h(x),h(y)\} }{\sqrt{h(x)h(y)}  }}
   \bigg)  \\&&\hspace{5cm}  - C( 2d_{CC}(\bar x,\bar y) -h(x)-h(y)) -\e \\
   &=& -2 \ln \bigg( 1 +\frac{\max\{h(x),h(y)\} }{d_{CC}(\bar x,\bar y)  }
   \bigg)  \\ &&\hspace{5cm} - C( 2d_{CC}(\bar x,\bar y) -h(x)-h(y)) -\e .
\end{eqnarray*}
The proof is then concluded by applying the same argument as in  \eqref{a-z}.
 \end{proof}

\section{Local biLipschitz equivalence of Bourdon functions and proof of main result}\label{loclip}

In this section we prove  Proposition~\ref{p2} and the main result, Theorem~\ref{main}.

 \begin{proof}[Proof of Proposition~\ref{p2}]  
 Let $\bar p$  as in the statement and choose $\e>0$ such that $\exp(\frac{3}{2}\e)\le 1+\bar \e$.
 Invoke Theorem~\ref{bb-better}  in correspondence to the choice of $\bar p $ and $\e$, to obtain the value  $r>0$ and select any 
      $\omega \in \p D\cap B(\bar p,r)\setminus \{\bar p\}$. In correspondence to this choice of $\omega$, Theorem~\ref{bb-better} yields  a smaller radius $0<r'<r$, 
so that if we  choose  $y\in D\cap B(\bar p,r')$ and $o\in D\cap B(\omega,r')$  and then apply Theorem~\ref{bb-better} to the quintuplet $(\bar p, \bar p, \omega, y,o)$ 
we obtain
 $$|g(y,o)-d_K(y,o)|<\eps, \qquad \text{ for all } y\in D\cap B(\bar p,r') ,\ \text{  and }  o\in D\cap B(\omega,r')$$

Next,  given $p,q\in \p D\cap B(\bar p, r')$ we similarly use
 Theorem ~\ref{bb-better}
to infer the existence of a  $r''>0$ for which, applying Theorem~\ref{bb-better} to the quintuplet $(\bar p,   p, q   x, y )$
 $$|d_K(x,y) -g(x,y) |\leq \eps , \qquad \text{ for all } 
  x\in D\cap B( p,r'')
  ,\, \text{ and for all } 
   y\in D\cap B(q,r'')
. $$ 
If $x_i$ (resp., $y_i$) is a sequence in $D$ converging to $p$ (resp., $q$), then for $i$ large enough
$   x_i\in D\cap B( p,r'')
$ and $
   y_i\in D\cap B(q,r'')$
   and
   $
    x_i ,
   y_i\in   B(\bar p,r')$. From the above bounds one obtains   
 \begin{eqnarray*}
\left|\langle y_i,x_i\rangle_o^g- \langle y_i,x_i\rangle_o^K \right|
&=&
  \frac{1}{2} \left| g(y_i,o)-d_K(y_i,o) + g(x_i,o)-d_K(x_i,o)+d_K(x_i,y_i) -g(x_i,y_i) \right| \\
  &\leq &   \frac{3}{2}\eps.
\end{eqnarray*} 
Consequently, if the sequences  $x_i$,   $y_i$ are taken so that $\langle p,q\rangle_o^g = \lim_{i\to \infty}  \langle y_i,x_i\rangle_o^g $, we have
 \begin{eqnarray*}
\dfrac{ \rho^K_o(p,q)}{\rho^g_o(p,q)}
&\leq& \frac{\lim_{i\to \infty} \exp(-\langle y_i,x_i\rangle_o^K)}{ \lim_{i\to \infty}\exp(-\langle y_i,x_i\rangle_o^g)}
\\
&=&\lim_{i\to \infty}  \exp\left(   \langle y_i,x_i\rangle_o^g- \langle y_i,x_i\rangle_o^K \right)
\\
&\leq &
\exp( \frac{3}{2}\eps) \le 1+\bar \e.
\end{eqnarray*}
And similarly, 
$ {\rho^g_o(p,q)}/{ \rho^K_o(p,q)} $ is bounded by 
 $1+
 \bar\eps$.
\end{proof}

 \begin{proof}[Proof of  Theorem~\ref{main}] 
    For any $\bar p\in  \p D_1$ and $\bar \e>0$ we show that the boundary extension is $(1+\bar \e)-$quasi-conformal at $\bar p$, i.e.  $ H^*(\bar p,F, d_{CC}, d_{CC})\le 1+\bar \e$,   where $H^*$ is defined as in \eqref{hstella}. 
    Following the  diagram \eqref{diagram} in the introduction, from \eqref{chain}  for every  $o\in D_1$  
    we have 
      \begin{multline}\label{dilations}
 H^*(\bar p,F, d_{CC}, d_{CC}) \\ \le 
 H^*(\bar p, {\rm Id}_{\p D_1}  , d_{\rm CC}, \rho^{g}_o)
 H^*(   \bar p, {\rm Id}_{\p D_1}  , \rho^g_o, \rho^K_o)
 H^*(\bar p, F, \rho^{K}_o, \rho^{K}_{f(o)})  \\ \cdot 
 H^*(F(\bar p), {\rm Id}_{\p D_2},\rho^K_{f(o)}, \rho^g_{f(o)} )
 H^*(F(\bar p),{\rm Id}_{\p D_2}, \rho^{g}_{f(o)}, d_{\rm CC}).
  \end{multline}

 Start by observing that   for any $o\in D_1$  the pointed metric spaces
 $(D_1, d_{K}, o)$ and $(D_2,d_{K}, f(o))$ are isometric. Thus they give rise to visual boundaries that are isometric with respect to the induced distances $\rho^{K}_o$ an $\rho^{K}_{f(o)}$, as defined in \eqref{def:Bourdon:dist}.  Consequently the induced extension map $F:(\p D_1, \rho^{K}_o) \to (\p D_2, \rho^{K}_{f(o)})$ is an isometry, and hence from     \eqref{biLip}
\begin{equation}\label{H3}H^*(\bar p, F, \rho^{K}_o, \rho^{K}_{f(o)}) =1.\end{equation}
%
%
%

 Regarding the first and last term in the right-hand side of \eqref{dilations},
  in view of Proposition~\ref{p1} we have that 
  \begin{equation}\label{H15}H^*(\bar p, {\rm Id}_{\p D_1}  , d_{\rm CC}, \rho^{g}_o)=H^*(F(\bar p),{\rm Id}_{\p D_2}, \rho^{g}_{f(o)}, d_{\rm CC})=1.
  \end{equation}

 We shall then prove that
\begin{equation}\label{H24} H^*(   \bar p, {\rm Id}_{\p D_1}  , \rho^g_o, \rho^K_o)
\leq 1+\bar \eps \, \text{ and } \,
 H^*(F(\bar p), {\rm Id}_{\p D_2},\rho^K_{f(o)}, \rho^g_{f(o)} ) \leq 1+\bar \eps,
 \end{equation}
 for some suitable choice of $o$.
 To prove this we will need to invoke Proposition~\ref{p2} twice, in $D_1$ and in $D_2$, together with the observation \eqref{biLip}. 
 Namely, we shall prove that
 for a suitable choice of $o$ 
 The maps considered in \eqref{H24} are $( 1+\bar \e )$-biLipschitz in a neighborhood of the considered points.

 First we apply Proposition~\ref{p2}  in a neighborhood of $F(\bar p) \in \p D_2$, thus yielding $r_2>0$ such that for all $\omega_2\in \p D_2\cap B(F(\bar p), r_2)\setminus \{ F(\bar p)\}$ there exists $r_2'>0$ such that for all  $o_2\in D_2 \cap B(\omega_2,r_2')$ one has that $\rho_{o_2}^g$ and $\rho_{o_2}^K$ are $(1+\bar \e)$-biLipschitz in $\p D_2\cap B(F(\bar p), r_2')$. For the moment we do not choose any specific $\omega_2$ and $o_2$, so $r_2'$ is still to be determined.

Next, we apply Proposition~\ref{p2}  to  $D_1$ in a neighborhood of $\bar p$ and use it to choose    $r_1>0$ such that   for all $\omega_1\in \p D_1\cap B(\bar p, r_1)\setminus \{ \bar p\}$ there exists $r_1'>0$ such that $o_1\in D_1 \cap B(\omega_1,r_1')$ one has that $\rho_{o_1}^g$ and $\rho_{o_1}^K$ are $(1+\bar \e)$-biLipschitz in $\p D_1\cap B(\bar p, r_1')$. 
By continuity of the map $F$ we may have chosen $r_1$ small enough that 
$F(B(\bar p, r_1)\cap D_1)\subset B(F(\bar p), r_2)\cap D_2$.

We set $\omega_2:=F(\omega_1)$, which is then in $ B(F(\bar p), r_2)\cap D_2$ and is different than $F(\bar p)$ since $F$ is a homeomorphism.
 Now we fix   $r_2'$ accordingly, as we explained above.  If needed we will select a smaller value for $r_1'$ so that we can assume $F(B(\omega_1,r_1')\cap D_1)\subset B(F(\omega_1), r_2')\cap D_2$.

To conclude, we can now select any base point $o\in B(\omega_1,r_1')\cap D_1$, so that $f(o)\in  B(\omega_2, r_2')\cap D_2$ and 
and hence
$\rho_{o_1}^g$ and $\rho_{o_1}^K$ are $(1+\bar \e)$-biLipschitz in $\p D_1\cap B(\bar p, r_1')$
and $\rho_{o_2}^g$ and $\rho_{o_2}^K$ are $(1+\bar \e)$-biLipschitz in $\p D_2\cap B(F(\bar p), r_2')$.
Thus, \eqref{biLip} gives \eqref{H24}.

Using the estimates \eqref{H3}, \eqref{H15}, and \eqref{H24} in \eqref{dilations} we get
$ H^*(\bar p,F, d_{CC}, d_{CC})  \le 1+\bar \eps$.
By the arbitrariness of $\bar \eps$ we deduce 
$H^*(\bar p,F, d_{CC}, d_{CC})  = 1$.
Finally, from Lemma \ref{lemma:CLO} we conclude.
  \end{proof}

 \bibliography{Capogna_LeDonne-BIBLIO}
\bibliographystyle{amsalpha}

\end{document}